\documentclass{amsart}
\usepackage{amssymb,amsmath,amsthm,hyperref,tikz,braids,t-angles,subcaption,booktabs,longtable}

\newtheorem{thm}{Theorem}[section]
\newtheorem{lem}[thm]{Lemma}
\newtheorem{cor}[thm]{Corollary}
\newtheorem{Thm}{Theorem}[section]

\newcommand{\s}{\sigma}
\newcommand{\Z}{\mathbb Z}
\newcommand{\Q}{\mathbb Q}
\newcommand{\R}{\mathbb R}
\newcommand{\C}{\mathbb C}

\allowdisplaybreaks

\begin{document}

\title{Some evaluations of Jones polynomials for certain families of weaving knots}
\date{\today}
\author{Sahil Joshi}
\address{Indian Institute of Technology Ropar, Rupnagar, Punjab 140001, India}
\email{2018maz0001@iitrpr.ac.in}

\author{Komal Negi}
\email{komal.20maz0004@iitrpr.ac.in}

\author{Madeti Prabhakar}
\email{prabhakar@iitrpr.ac.in}

\begin{abstract}
	In this paper, we derive formulae for the determinant of weaving knots $W(3,n)$ and $W(p,2)$.
	We calculate the dimension of the first homology group with coefficients in $\Z_3$ of the double cyclic cover of the 3-sphere $S^3$ branched over $W(3,n)$ and $W(p,2)$ respectively.
	As a consequence, we obtain a lower bound of the unknotting number of $W(3,n)$ for certain values of $n$.
\end{abstract}

\maketitle
\tableofcontents

\section{Introduction}

Let $B_{N+1}$ denote the Artin braid group generated by $N$ generators, namely, $\s_1,\s_2,\ldots,\s_N$ satisfying the braid relations:
\begin{equation*}
	\begin{array}{llll}
		\mathrm{(B1)} & \s_i \s_j = \s_j \s_i, && (i,j\in\{1,2,\ldots,N\}, |i-j|\geq 2),\\
		\mathrm{(B2)} & \s_i \s_{i+1} \s_i = \s_{i+1} \s_i \s_{i+1}, && (i\in\{1,2,\ldots,N-1\}).
	\end{array}
\end{equation*}

There is a well-known geometric interpretation of the group $B_{N+1}$ in terms of braid diagrams on $(N+1)$-strands.
Throughout the work, the braid diagrams which correspond to the generator $\s_i$ and its inverse $\s_i^{-1}$ will be
\begin{align*}
	&\raisebox{-0.15cm}{$\mathrm{(i)~} \s_i \leftrightarrow$~}
	\begin{tangles}{ccccccc}
		{\scriptstyle 1} & & {\scriptstyle i-1} & {\scriptstyle ~~i\quad i+1~} & {\scriptstyle i+2} & & {\scriptstyle N+1}\\
		&\vspace{2pt}\\
		\id&\raisebox{0.3cm}{$~\ldots$}&\id&\xx&\id&\raisebox{0.3cm}{$\ldots$}&\id
	\end{tangles}
	\raisebox{-0.15cm}{;}&
	&\raisebox{-0.15cm}{$\mathrm{(ii)~} \s_i^{-1} \leftrightarrow$~}
	\begin{tangles}{ccccccc}
		{\scriptstyle 1} & & {\scriptstyle i-1} & {\scriptstyle ~~i\quad i+1~} & {\scriptstyle i+2} & & {\scriptstyle N+1}\\
		&\vspace{2pt}\\
		\id&\raisebox{0.3cm}{$~\ldots$}&\id&\x&\id&\raisebox{0.3cm}{$\ldots$}&\id
	\end{tangles}
	\raisebox{-0.15cm}{.}
\end{align*}

Put $B_\infty = \bigcup_{N\geq 1}B_{N+1}$.
The closure of a braid (or its corresponding braid diagram) $\alpha\in B_\infty$, denoted by $\mathbf{cl}(\alpha)$, is obtained by joining the upper ends of its strands to the respective lower ends via non-intersecting curves in the plane.
This yields a knot/link diagram.
One of the remarkable theorems in knot theory is attributed to J.\,W.\,Alexander~(see Murasugi~\cite{Mur}), which states that every knot or link in the $3$-sphere $S^3$ can be presented by the closure of a braid diagram.
This connotes that braids are closely related to knots and links in $S^3$.
In fact, many families of knots are defined by using braids, for instance, weaving knots which are objects of our interest.

Let $p$ and $n$ be positive integers such that $p\geq 2$.
The \emph{weaving knot of type $(p,n)$}, denoted by $W(p,n)$, is an alternating knot or link which is the closure of the braid
$B_W(p,n)=\big(\s_1\s_2^{-1}\s_3\s_4^{-1}\cdots\s_{p-1}^{(-1)^p}\big)^n$ on $p$ strands.
For example, Figure~\ref{Wbraids} shows braids $B_W(p,n)$ for $(p,n) = (3,2), (4,2), (5,2)$, and $(4,3)$ whose closures represent weaving knots $W(3,2)=4_1, W(4,2)=6^2_3, W(5,2)=8_{12}$, and $W(4,3)=9_{40}$, respectively, as per the Alexander-Briggs-Rolfsen knot table (see~Rolfsen~\cite[Appendix~C]{Rol}).

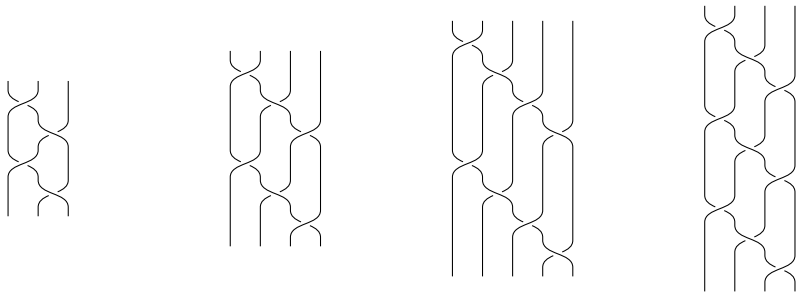
\begin{figure}[h!]
	\begin{subfigure}{0.24\textwidth}
		\centering
		\begin{tikzpicture}[scale=0.4]
		\braid s_1^{-1} s_2 s_1^{-1} s_2;
		\end{tikzpicture}
	\end{subfigure}
	\begin{subfigure}{0.24\textwidth}
		\centering
		\begin{tikzpicture}[scale=0.4]
		\braid s_1^{-1} s_2 s_3^{-1} s_1^{-1} s_2 s_3^{-1};
		\end{tikzpicture}
	\end{subfigure}
	\begin{subfigure}{0.24\textwidth}
		\centering
		\begin{tikzpicture}[scale=0.4]
		\braid s_1^{-1} s_2 s_3^{-1} s_4 s_1^{-1} s_2 s_3^{-1} s_4;
		\end{tikzpicture}
	\end{subfigure}
	\begin{subfigure}{0.24\textwidth}
		\centering
		\begin{tikzpicture}[scale=0.4]
		\braid s_1^{-1} s_2 s_3^{-1} s_1^{-1} s_2 s_3^{-1} s_1^{-1} s_2 s_3^{-1};
		\end{tikzpicture}
	\end{subfigure}
	\caption{Braids $B_W(3,2), B_W(4,2), B_W(5,2)$, and $B_W(4,3)$, respectively.}
	\label{Wbraids}
\end{figure}

Weaving knots possess several distinguishing features.
They share the same projection with torus knots and links, another important class of knots and links.
They are hyperbolic, alternating, and conjectured to have the largest hyperbolic volume per crossing number.
Moreover, every knot or link is obtained from a weaving knot's standard closed braid presentation by switching some of its crossings.
An extensive study of the family of weaving knots will provide an insight to `volume conjecture'~(Singh, Mishra and Ramadevi~\cite[page~2]{SMR}).
Therefore weaving knots are among the prominent objects in knots and links.
For substantial literature on weaving knots, one may refer to the research articles \cite{SMR}, Mishra and Staffeldt~\cite{MS}, and Champanerkar, Kofman and Purcell~\cite{CKP}.

By Alexander's theorem, the map $\mathbf{cl}:B_\infty\rightarrow\{\text{knots/links}\}$ is surjective.
It is known that the correspondence $\alpha \mapsto \mathbf{cl}(\alpha)$ is many-to-one, i.e., the respective closures of two different braid diagrams may represent the same knot type.
In this regard, a characterization is given by a theorem of A.\,A.\,Markov~(see~\cite{Mur}), which asserts that any two braids are Markov equivalent if and only if their closures represent isotopic knots/links. 

The notion of Markov equivalence is as follows:
The two types of Markov moves on braids from the set $B_\infty$ will be denoted by the following:
\begin{equation*}
	\begin{array}{lll}
		\mathrm{(M1)} & \beta \rightarrow \gamma\beta\gamma^{-1},
		\text{or, } \gamma\beta\gamma^{-1} \rightarrow \beta, & (\beta,\gamma\in B_{N+1}),\\
		\mathrm{(M2)} & \beta \rightarrow \beta\s_{N+1}^{\pm1}, \text{or, } \beta\s_{N+1}^{\pm1} \rightarrow \beta, & (\beta\in B_{N+1}\hookrightarrow B_{N+2},\s_{N+1}\in B_{N+2}).\\
	\end{array}
\end{equation*}

Two braids $\alpha,\beta \in B_\infty$ are said to be Markov equivalent if there exists a finite sequence of Markov moves transforming $\alpha$ into $\beta$ considered upto braid moves \textrm{(B1)} and \textrm{(B2)}.

Alexander and Markov theorems, put together, translate the topological study of knots and links into an algebraic study of Markov equivalent braids.
In fact, several invariants of knots and links are defined using braids and braid groups, for instance, the Jones polynomial.

Suppose that $L$ is an oriented link in the $3$-sphere $S^3$ with $\mu(L)$ components.
The values of the Jones polynomial $V_L(t) \in \Z[t^{\pm\frac{1}{2}}]$ of the link $L$ at $t=e^{2\pi i/n}$ for $n = 1,2,3,4,6,10$ are special because of their relationship with other isotopy invariants of $L$.
For details, see Jones~\cite[$\S\,12$]{Jon}.

The main objective of this paper is to examine Jones polynomials for two infinite subfamilies, namely $\{W(p,2): p\geq 2\}$ and $\{W(3,n): n\geq 1\}$, of the doubly infinite family of weaving knots $\{W(p,n): p\geq 2, n\geq 1\}$.

Jones~\cite{Jon} showed that $V_L(-1) = \Delta_L(-1)$, where $\Delta_L(t) \in \Z[t^{\pm\frac{1}{2}}]$ is the Alexander polynomial of $L$.
The absolute value of $\Delta_L(-1)$, denoted by $\det(L)$, is called determinant of the link $L$.
Let $D_L$ be the double cyclic cover of $S^3$ branched over $L$ and let $H_1(D_L;\Z_3)$ be the first homology group of $D_L$ with coefficients in $\Z_3$.
Let $n_L$ denote the dimension of the vector space $H_1(D_L;\Z_3)$ over the field $\Z_3$.
The following theorem is due to Lickorish and Millett~\cite{LM}.

\begin{Thm}[{\cite[Theorem 3]{LM}}]\label{LickMill}
	$V_L(e^{i\pi/3}) = \pm\,i^{\mu(L)-1}\,(i\sqrt{3})^{n_L}$.
\end{Thm}

Further, the value of $n_L$ is also related to the unknotting number.
The \emph{unknotting number} of a knot $K$, denoted by $u(K)$, is the minimum number of crossing changes required to transform $K$ into the trivial knot $0_1$.
For example, $u(3_1)=1$, $u(8_{12})=2$, $u(8_{19})=3$, and $u(T(p,q))=\frac{(p-1)(q-1)}{2}$, where $T(p,q)$ denotes the torus knot of type $(p,q)$.
Let $K$ be a knot whose unknotting number is $u(K)$ and $\dim H_1(D_K;\Z_3) = n_K$.
The following theorem of Wendt is given in Miyazawa~\cite{Miy}.
\begin{Thm}[{\cite[Corollary 1.4]{Miy}}]\label{Miya}
	$u(K)\geq n_K$.
\end{Thm}

It seems unlikely that new information about unknotting numbers can be obtained from $V_L(e^{i\pi/3})$, though calculation of these may give a quick way of computing $n_L$~(\cite[page~351]{LM}).
However, Traczyk~\cite{Tra} and later Stoimenow~\cite{Sto} have demonstrated how evaluations of link polynomials can be used to determine unknotting numbers for some of the knots.

In this paper, we shall evaluate the Jones polynomial $V_L(t)$ of $L = W(3,n)$, $W(p,2)$ at $t=e^{i\pi/3},-1$.
These specific values are considered in order to obtain information on the unknotting number and the knot determinant for these two families of weaving knots.

Mishra \& Staffeldt~\cite{MS} obtained the formula
\begin{equation}\label{eq:JonesW3n}
	\begin{split}
		V_{W(3,n)}(t) & = t^{-n-1} \left[ (1+t)^2 C_{n,0}(t) + (1+t)
		\left(C_{n,1}(t) + C_{n,2}(t)\right) t^2 + \right.\\
		& \qquad\left. \left(C_{n,12}(t) + C_{n,21}(t)\right) t^4 \right],
	\end{split}
\end{equation}
where the polynomials $C_{n,0}(t), C_{n,1}(t), C_{n,2}(t), C_{n,12}(t), C_{n,21}(t) \in \Z[t]$
have a recursive definition ---
\begin{subequations}\label{eq:oldJonesW3n}
	\begin{align}
		C_{n,0}(t)  &= -t(t-1) C_{n-1,1}(t) + t^2 C_{n-1,21}(t),\\
		C_{n,1}(t)  &= -(t-1) C_{n-1,0}(t) - (t-1)^2 C_{n-1,1}(t),\\
		C_{n,2}(t)  &= t C_{n-1,1}(t),\\
		C_{n,12}(t) &= C_{n-1,0}(t) + (t-1) C_{n-1,1}(t),\\
		C_{n,21}(t) &= t C_{n-1,12}(t)
	\end{align}
\end{subequations}
--- which we have reformulated in terms of matrices as:
\begin{Thm}[{cf \cite[page~31]{MS}}]\label{JonesW3n}
	Let
	\begin{equation*}
		C_n(t) = \begin{bmatrix}
					C_{n,0}(t) & C_{n,1}(t) & C_{n,2}(t) & C_{n,12}(t) & C_{n,21}(t)
		\end{bmatrix}^T,
	\end{equation*}
	where $n$ is any positive integer, $C_{n,0}(t), C_{n,1}(t), C_{n,2}(t), C_{n,12}(t), C_{n,21}(t) \in \Z[t]$, and
	\begin{equation*}
		C_1(t) = \begin{bmatrix}
					0 & -(t-1) & 0 & 1 & 0
		\end{bmatrix}^T.
	\end{equation*}
	For $n\geq 2$, suppose that $C_n(t) = M(t)\,C_{n-1}(t)$, where
	\begin{equation*}
		M(t) = \begin{bmatrix}
					0 & -t(t-1) & 0 & 0 & t^2\\
					-(t-1) & -(t-1)^2 & 0 & 0 & 0\\
					0 & t & 0 & 0 & 0\\
					1 & t-1 & 0 & 0 & 0\\
					0 & 0 & 0 & t & 0
		\end{bmatrix}.
	\end{equation*}
	Then the Jones polynomial $V_{W(3,n)}(t)$ of the weaving knot $W(3,n)$ is given by~\eqref{eq:JonesW3n}.
\end{Thm}

Using Theorem~\ref{JonesW3n}, rewrite~\eqref{eq:JonesW3n} as:
\begin{equation}\label{eq:newJonesW3n}
	V_{W(3,n)}(t) = t^{-n-1}\,Z(t)\,C_n(t) = Z(t)\left(t^{-n-1} M^{n-1}(t)\right) C_1(t),
\end{equation}
where the matrix
\begin{equation*}
	Z(t) = \left[\begin{matrix}
		(1+t)^2 & (1+t) t^2 & (1+t) t^2 & t^4 & t^4
	\end{matrix}\right].
\end{equation*}

In this paper, we use Theorem~\ref{JonesW3n} to calculate the dimension of the vector space $H_1(D_{W(3,n)};\Z_3)$ over the field $\Z_3$, which in turn yields a lower bound of the unknotting number of $W(3,4k)$, see Corollary~\ref{nW3n}.
By using Theorem~\ref{JonesW3n} again, we derive a knot determinant formula for $W(3,n)$ in Theorem~\ref{detW3n}.
Further, we give a recursive formula for the Jones polynomial of $W(p,2)$ in Theorem~\ref{JonesWp2}.
As applications of Theorem~\ref{JonesWp2}, we calculate the dimension of the vector space $H_1(D_{W(p,2)};\Z_3)$ over $\Z_3$ and derive a knot determinant formula for $W(p,2)$, see Corollary~\ref{nWp2} and Theorem~\ref{detWp2}, respectively.

\section{Some invariants of the weaving knot $W(3,n)$}

\subsection{The dimension of the vector space $H_1(D_{W(3,n)};\Z_3)$ over $\Z_3$}

Let $w=e^{i\pi/3}$ which satisfies the irreducible polynomial $f(x) = x^2-x+1$ in the polynomial ring $\Q[x]$.
We consider the subfield $\frac{\Q[x]}{\langle x^2-x+1\rangle} \cong \Q(w)$ 
of $\C$ to work with matrices $C_n = C_n(w), M = M(w)$, and $Z = Z(w)$ whose entries are from the field $\Q(w)$. 

Put $A_n = w^{-n-1}M^{n-1}$.
Then \eqref{eq:newJonesW3n} reduces to $V_{W(3,n)}(w) = ZA_nC_1$.

\begin{lem}\label{polyM}
	If $g(x)=x^6+w x^2 \in \Q(w)[x]$, then the matrix $M$ satisfies the polynomial $g(x)$.
\end{lem}
\begin{proof}
	Let $h(x)\in \Q(w)[x]$ be the characteristic polynomial of the matrix $M$.
	Then
	\begin{align*}
		h(x) & = \det(M-xI)\\
		& = \left|\begin{matrix}
			-x & 1 & 0 & 0 & w-1\\
			-w+1 & w-x & 0 & 0 & 0\\
			0 & w & -x & 0 & 0\\
			1 & w-1 & 0 & -x & 0\\
			0 & 0 & 0 & w & -x
		\end{matrix}\right|\\
		& = -x^5 + wx^4 +(1-w)x^3 - x^2.
	\end{align*}
	Thus $h(M) = \mathbf{0}$.
	Now
	\begin{align*}
		g(x) & = x^6 + wx^2\\
		& = x \left(wx^4 + (1-w)x^3 - x^2 - h(x)\right) + wx^2\\
		& = wx^5 + (1-w)x^4 - x^3 + wx^2 - xh(x)\\
		& = w \left(wx^4 + (1-w)x^3 - x^2 - h(x)\right) + (1-w)x^4 - x^3 + wx^2 - xh(x)\\
		& = \left(w^2-w+1\right) x^4 - \left(w^2-w+1\right) x^3 -wx^2 + wx^2 - (x+w)h(x)\\
		& = -(x+w)h(x).
	\end{align*}
	Hence $g(M) = M^6 + wM^2 = -(M+wI)\,h(M) = \mathbf{0}$.
\end{proof}

\begin{lem}\label{Ai}
	For every integer $n\geq 1$, we have $A_{3+4n} = A_3$.
	Hence $A_{4+4n} = A_4$, $A_{5+4n} = A_5$, and $A_{6+4n} = A_6$ for each integer $n\geq 1$.
\end{lem}
\begin{proof}
	We shall use mathematical induction.
	For $n=1$, we have $A_{7} = w^{-8}M^6$.
	By Lemma~\ref{polyM}, $w^{-8}M^6 = w^{-2}(-wM^2) = -w^{-1}M^2 = w^{-4}M^2 = A_3$.
	
	Assume that $A_{3+4k} = w^{-3-4k-1}M^{3+4k-1} = w^{-4-4k}M^{4k+2} = A_3$, for some integer $k\geq 2$.
	Using Lemma~\ref{polyM} and our induction hypothesis, we get
	\begin{align*}
		A_{3+4(k+1)} & = w^{-3-4(k+1)-1}M^{3+4(k+1)-1}\\
		& = w^{-8+4k}M^{4k+6}
		= w^{-2+4k}M^{4k}(-wM^2)\\
		& = -w^{-1+4k}M^{4k+2}
		= w^{-4+4k}M^{4k+2}
		= A_{3+4k} = A_3.
	\end{align*}
	This completes the proof of the fact that $A_3 = A_7 = A_{11} = A_{15} =\cdots$.
	The remaining part follows directly after multiplying $A_{3+4n} = A_3$ by $w^{-1}M, w^{-2}M^2$, and $w^{-3}M^3$ respectively.
\end{proof}

Due to Lemma~\ref{Ai}, we only need to know matrices $A_1,A_2,\ldots,A_6$ to find all the elements of the sequence of matrices $\{A_i\}$.

\begin{thm}\label{eJonesW3n}
	For the weaving knot $W(3,n)$, the value of its Jones polynomial at $t=w$ is given by
	\begin{equation}\label{eq:eJonesW3n}
		V_{W(3,n)}(w) = \begin{cases}
			3, & \text{if } n = 4k, \text{ where } k\geq 1,\\
			-1, & \text{if } n = 4k-2, \text{ where } k\geq 1,\\
			1, & \text{otherwise}.
		\end{cases} 
	\end{equation}
\end{thm}
\begin{proof}
	Let $k$ be any positive integer.
	One can calculate
	\begin{align*}
		V_{W(3,4k)}(w) & = Z A_{4k} C_1\\
		& = Z A_4 C_1\\
		& = \begin{bmatrix}
			3w & w-2 & w-2 & -w & -w
		\end{bmatrix}\times\\
		&\qquad\quad \left(w^{-5}
		\begin{bmatrix}
			0 & 1 & 0 & 0 & w-1\\
			-w+1 & w & 0 & 0 & 0\\
			0 & w & 0 & 0 & 0\\
			1 & w-1 & 0 & 0 & 0\\
			0 & 0 & 0 & w & 0
		\end{bmatrix}^3\right)\times
		\begin{bmatrix}
			0\\ -w+1\\ 0\\ 1\\ 0
		\end{bmatrix}\\
		& = 3.
	\end{align*}
	Similarly, one can obtain $V_{W(3,4k-2)}(w) = Z A_{4k-2} C_1 = Z A_6 C_1 = -1$ for every $k\geq 2$ and $V_{W(3,n)}(w) = 1$, for every $n$ such that $n = 4k-1, 4k-3$.
	It can be checked directly that $V_{W(3,2)}(w) = Z A_{2} C_1 = -1$.
	This completes the proof.
\end{proof}

\begin{cor}\label{nW3n}
	Suppose that $k$ is any positive integer.
	For the weaving knot $W(3,n)$, $n_{W(3,n)} = \dim H_1(D_{W(3,n)};\Z_3) = 0$ except for the case when $n=4k$.
	If $n=4k$, then $n_{W(3,n)} = \dim H_1(D_{W(3,n)};\Z_3) = 2$.
	Hence if $\gcd(3,4k) = 1$, then $u(W(3,4k)) \geq 2$.
\end{cor}
\begin{proof}
	The proof of $n_{W(3,n)} = \dim H_1(D_{W(3,n)};\Z_3) = 0$ when $n\neq 4k$ follows immediately from Theorem~\ref{eJonesW3n} and Theorem~\ref{LickMill}.
	If $n=4k$ and $\gcd(3,4k) = 1$, then $\mu(W(3,4k))$ is $1$.
	By Theorem~\ref{eJonesW3n}, we obtain that $n_{W(3,4k)}=2$ and by Theorem~\ref{Miya}, we deduce that $u(W(3,4k)) \geq 2$.
\end{proof}

It is known that $u(8_{18})=2$, where $8_{18}=W(3,4)$.
This fact can also be verified as follows: By Corollary~\ref{nW3n}, $u(8_{18}) \geq n_{8_{18}} = 2$.
Further,
\begin{align*}
	\mathbf{cl}(W(3,4)) &=
	\mathbf{cl}(\s_1 \s_2^{-1} \s_1 \s_2^{-1} \s_1 \s_2^{-1} \s_1 \s_2^{-1}) \rightarrow
	\mathbf{cl}(\s_1 \s_2 \s_1 \s_2^{-1} \s_1 \s_2^{-1} \s_1 \s_2^{-1})\\
	&\quad\rightarrow
	\mathbf{cl}(\s_1 \s_2 \s_1 \s_2^{-1} \s_1^{-1} \s_2^{-1} \s_1 \s_2^{-1}) =
	\mathbf{cl}(\s_1 \s_2^{-1}) = 0_1
\end{align*}
is an unknotting sequence that transforms $W(3,4)$ into $0_1$ by $2$ crossing changes.
Thus $u(W(3,4)) \leq 2$.
Hence $u(W(3,4)) = 2$.

This seems intriguing that the Jones polynomial detects the property that $u(W(3,4k))\neq 1$ for every positive integer $k$.
However, it fails to provide any information on the unknotting number of $W(3,n)$, as $n_{W(3,n)}=0$, when $n\neq 4k$.

\subsection{A formula for knot determinant of $W(3,n)$}

On substituting $t=-1$ in~\eqref{eq:JonesW3n}, we obtain that
\begin{equation*}
	\vert V_{W(3,n)}(-1) \vert = \vert C_{n,12}(-1) + C_{n,21}(-1) \vert = \det (W(3,n)).
\end{equation*}

It is clear from \eqref{eq:oldJonesW3n} that for any $n$, none of the $C_{n,12}(t), C_{n,21}(t), C_{n,1}(t),$ and $C_{n,0}(t)$ depend on $C_{_{-},2}(t)$.
Henceforth, we exclude $C_{_{-},2}(t)$ from our calculations.

\begin{thm}\label{detW3n}
	Let $\varphi = \frac{1 + \sqrt{5}}{2}$.
	Then the determinant of the weaving knot $W(3,n)$ can be written as
	\begin{equation}\label{eq:detW3n}
		\det (W(3,n)) = -2 + (1 + \varphi)^n + (1 - \varphi^{-1})^n.
	\end{equation}
\end{thm}

\begin{proof}
	Let $\widetilde Z, \widetilde M$, and $\widetilde{C_1}$ be the matrices obtained from the matrices $Z(-1)$, $M(-1)$, and $C_1(-1)$ by deleting their third column, third row and third column, and third row, respectively.
	Thus, $C_{_{-},2}(t)$ is eliminated and we have
	\begin{align*}
		\widetilde Z & = \begin{bmatrix}
			0 & 0 & 1 & 1
		\end{bmatrix},&
		\widetilde M & = \begin{bmatrix}
			0 & -2 & 0 & -1\\
			2 & -4 & 0 & 0\\
			1 & -2 & 0 & 0\\
			0 & 0 & -1 & 0
		\end{bmatrix},&
		\widetilde{C_1} & = \begin{bmatrix}
			0\\
			2\\
			1\\
			0
		\end{bmatrix}.
	\end{align*}
	The characteristic polynomial of the matrix $\widetilde M$ is $f_{\widetilde M}(x) = x^4 + 4x^3 + 4x^2 + x$.
	The eigenvalues of $\widetilde M$ are $0, -1, -\frac{3+\sqrt{5}}{2}, -\frac{3-\sqrt{5}}{2}$.
	Clearly, $\widetilde M$ is diagonalizable over $\R$ and $\widetilde M=PDP^{-1}$, where
	\begin{align*}
		P & = \begin{bmatrix}
			3 & 2 & \frac{5 + \sqrt{5}}{2} & \frac{5 - \sqrt{5}}{2}\\
			2 & 1 & 3 + \sqrt{5} & 3 - \sqrt{5}\\
			1 & 0 & \frac{3 + \sqrt{5}}{2} & \frac{3 - \sqrt{5}}{2}\\
			1 & 2 & 1 & 1
		\end{bmatrix},&
		D & = \begin{bmatrix}
			-1 & 0 & 0 & 0\\
			0 & 0 & 0 & 0\\
			0 & 0 & -\left( \frac{3 + \sqrt{5}}{2} \right) & 0\\
			0 & 0 & 0 & -\left( \frac{3 - \sqrt{5}}{2} \right)
		\end{bmatrix},
	\end{align*}
	\begin{align*}
		P^{-1} & = \begin{bmatrix}
			1 & 0 & -1 & -1\\
			0 & 1 & -2 & 0\\
			\frac{\sqrt 5 - 5}{10} & \frac{3}{\sqrt 5} - 1 & \frac{5}{2} - \frac{11}{2\sqrt 5} & 1 - \frac{2}{\sqrt 5}\\
			\frac{-\sqrt 5 - 5}{10} & -\frac{3}{\sqrt 5} - 1 & \frac{5}{2} + \frac{11}{2\sqrt 5} & 1+\frac{2}{\sqrt 5}
		\end{bmatrix}.
	\end{align*}
	After substituting $t=-1$ in~\eqref{eq:newJonesW3n}, we obtain
	\begin{align*}
		\det (W(3,n)) & = \left\vert \widetilde Z \left((-1)^{-n-1} \widetilde M^{n-1}\right) \widetilde{C_1} \right\vert\\
		& = \left\vert \widetilde Z \left(-\widetilde M\right)^{n-1} \widetilde{C_1} \right\vert\\
		& = \left\vert \left(\widetilde Z P\right) \left(-D\right)^{n-1} \left(P^{-1} \widetilde{C_1} \right) \right\vert\\
		& = \begin{bmatrix}
			2 & 2 & \frac{5 + \sqrt{5}}{2} & \frac{5 - \sqrt{5}}{2}\\
		\end{bmatrix}
		\begin{bmatrix}
			1 & 0 & 0 & 0\\
			0 & 0 & 0 & 0\\
			0 & 0 & \frac{3 + \sqrt{5}}{2} & 0\\
			0 & 0 & 0 & \frac{3 - \sqrt{5}}{2}
		\end{bmatrix}^{n-1}
		\begin{bmatrix}
			-1\\
			0\\
			\frac{5 + \sqrt{5}}{10}\\
			\frac{5 - \sqrt{5}}{10}
		\end{bmatrix}\\
		& = -2 + \left(\frac{3+\sqrt{5}}{2}\right)^n + \left(\frac{3-\sqrt{5}}{2}\right)^n.
	\end{align*}
	This completes the proof.
\end{proof}

\section{Some invariants of the weaving knot $W(p,2)$}

\subsection{A recursive formula for the Jones polynomial of $W(p,2)$}

A closed-form formula for the Jones polynomial of $W(p,2)$ is not known.
Nevertheless, a recursive formula for the same can be given as follows:

\begin{thm}\label{JonesWp2}
	For the weaving knot $W(p,2)$ where $p\geq 2$, the Jones polynomial is recursively defined by the equations:
	\begin{align*}
		V_{W(2,2)}(t) & = -(t^{\frac{5}{2}} + t^{\frac{1}{2}}),\\
		V_{W(3,2)}(t) & = t^{-2} - t^{-1}z V_{W(2,2)}(t),\quad\text{where } z=t^{\frac{1}{2}}-t^{-\frac{1}{2}},
	\end{align*}
	and for any integer $n\geq 2$,
	\begin{align*}
		V_{W(2n,2)}(t) & = t^2 V_{W(2n-2,2)}(t) + tz V_{W(2n-1,2)}(t),\\
		V_{W(2n+1,2)}(t) & = t^{-2} V_{W(2n-1,2)}(t) - t^{-1}z V_{W(2n,2)}(t).
	\end{align*}
\end{thm}
\begin{proof}
	The proof follows directly by applying the skein relation satisfied by the Jones polynomial (see~\cite{Mur}) over the skein tree diagram of $W(2n+1,2)$ shown in Figure \ref{skeintree}.
\end{proof}

\begin{figure}[h!]
	\centerline{\includegraphics[scale=0.75]{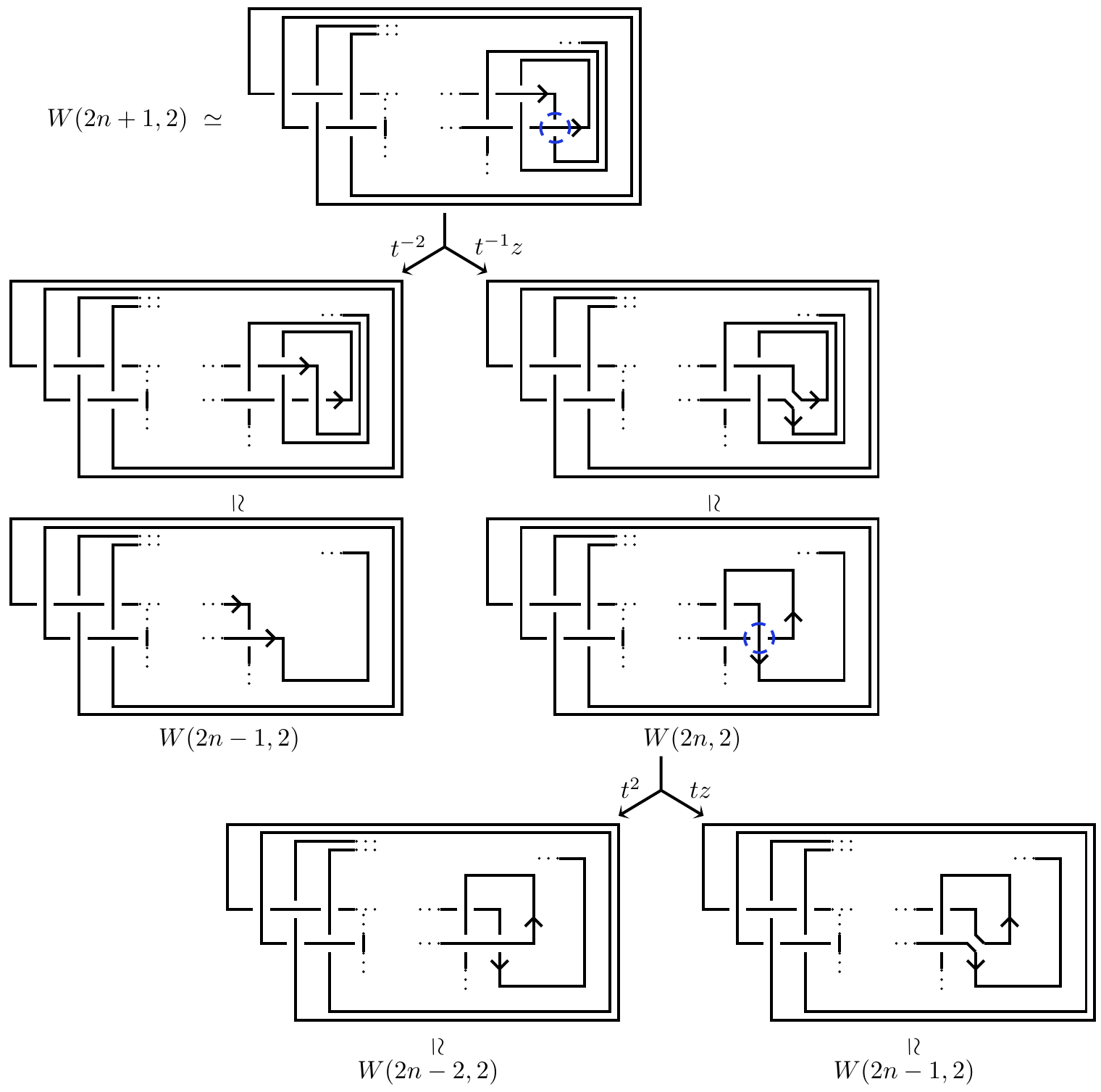}}
	\caption{A skein tree diagram of $W(2n+1,2)$.}\label{skeintree}
\end{figure}

In the sequel, we give some applications of Theorem~\ref{JonesWp2}.
\subsection{The dimension of the vector space $H_1(D_{W(p,2)};\Z_3)$ over $\Z_3$}
\begin{thm}\label{eJonesWp2}
	Let $v_{2n} = V_{W(2n,2)}(e^{i\pi/3})$ and $v_{2n+1} = V_{W(2n+1,2)}(e^{i\pi/3})$.
	Then
	\begin{align}
		v_{2n} & = \begin{cases}
			(-1)^k\,i, &\text{if } n=2k-1,~\text{where }k=1,2,3,\ldots,\\
			(-1)^{k+1}\,\sqrt{3}, &\text{if } n=2k,~\text{where }k=1,2,3,\ldots.
		\end{cases} \label{eq:1eJonesWp2}\\
		v_{2n+1} & = \begin{cases}
			-1, &\text{if } n\equiv 1,2\pmod 4,\\
			1, &\text{if } n\equiv 0,3\pmod 4.
		\end{cases} \label{eq:2eJonesWp2}
	\end{align}
\end{thm}
\begin{proof}
	Substitute $t=w=e^{i\pi/3}$ in Theorem~\ref{JonesWp2}.
	Then $z=i$ and for $n\geq 3$,
	\begin{align*}
		v_{2n+1} & = w^{-2} v_{2n-1} - iw^{-1} v_{2n}\\
		& = w^{-2} (w^{-2}v_{2n-3} - iw^{-1}v_{2n-2}) - iw^{-1} (w^2 v_{2n-2} + iw\,v_{2n-1})\\
		& = w^{-4}v_{2n-3} - (iw^{-3}+iw)v_{2n-2} - i^2 (w^{-2}v_{2n-3} - iw^{-1}v_{2n-2})\\
		& = (w^{-4}-i^2w^{-2}) v_{2n-3} + (i^3w^{-1}-iw-iw^{-3}) v_{2n-2}\\
		& = (w^{-2}-w^{-1}) v_{2n-3} + (-iw^{-1}-iw+i)v_{2n-2}\\
		& = -v_{2n-3}.
	\end{align*}
	Similarly for $n\geq 3$,
	\begin{align*}
		v_{2n} & = w^2 (w^2 v_{2n-4} + iw\,v_{2n-3}) + iw(w^{-2}v_{2n-3} - iw^{-1}v_{2n-2})\\
		& = w^4\,v_{2n-4} + (iw^3+iw^{-1})v_{2n-3} - i^2 (w^2 v_{2n-4} + iw\,v_{2n-3})\\
		& = (w^4-i^2w^2) v_{2n-4} + (-i^3w+iw^{-1}+iw^3) v_{2n-3}\\
		& = (w^2-w) v_{2n-4} + (iw+iw^{-1}-i)v_{2n-2}\\
		& = -v_{2n-4}.
	\end{align*}
	Since $v_2 = -i, v_3 = -1, v_4 = \sqrt{3}$, and $v_5 = -1$, we obtain~\eqref{eq:1eJonesWp2}~and~\eqref{eq:2eJonesWp2}.
\end{proof}

\begin{cor}\label{nWp2}
	For the weaving knot $W(p,2)$, the dimension $n_{W(p,2)}$ of the vector space $H_1(D_{W(p,2)};\Z_3)$ is given by
	\begin{equation*}
		n_{W(p,2)}  = \begin{cases}
			1, & \text{if } p=4k,\\
			0, & \text{otherwise}.
		\end{cases}
	\end{equation*}
\end{cor}
\begin{proof}
	The proof follows directly from Theorem~\ref{eJonesWp2} and Theorem~\ref{LickMill}.
\end{proof}

From Corollary~\ref{nWp2}, it is obvious that the value of $n_{W(p,2)}$ conveys nothing significant about the unknotting number of $W(p,2)$.

It is well-known that the unknotting number of a knot is less than or equal to half of its crossing number.
Thus $u(W(2n+1,2)) \leq \frac{c(W(2n+1,2))}{2} = 2n$.
Whereas lower bounds of unknotting numbers matter the most, it will also be interesting to search for a better upper bound.
Here we show that $u(W(2n+1,2)) \leq n$.
It has been observed in Figure~\ref{skeintree} that $W(2n-1,2)$ is obtained from $W(2n+1,2)$ by one crossing change.
This gives an unknotting sequence $W(2n+1,2) \rightarrow W(2n-1,2) \rightarrow W(2n-3,2) \rightarrow\cdots\rightarrow W(3,2) \rightarrow 0_1$ which transforms $W(2n+1,2)$ into the unknot $0_1$ by changing $n$ crossings.

The efficacy of this upper bound is not known to the authors.
It is known that $W(7,2) = 12a477$ and $u(12a477) = 2$ or $3$ which is anyway $\leq 3$.
In this regard, we pose the question: Find an integer $n$ such that $u(W(2n+1,2)) < n$.

\subsection{A formula for knot determinant of $W(p,2)$}

Here we shall use Theorem~\ref{JonesWp2} to derive a formula for the determinant of weaving knot $W(p,2)$.

\begin{thm}\label{detWp2}
	For any integer $p\geq 2$, the knot determinant of the weaving knot $W(p,2)$ is given by
	\begin{equation}\label{eq:detWp2}
		\det (W(p,2)) = \begin{cases}
			\begin{split}
				\frac{1}{2\sqrt{2}}
				\left[ \left(3 + 2\sqrt{2}\right)^n - \left(3 - 2\sqrt{2}\right)^n \right],
			\end{split}
			& \text{if } p=2n,\vspace{5pt}\\
			\begin{split}
				\frac{1}{4}
				& \left[ \left(2+\sqrt{2}\right) \left(3 + 2\sqrt{2}\right)^n + \right.\\
				& \qquad\quad \left. \left(2-\sqrt{2}\right) \left(3 - 2\sqrt{2}\right)^n \right],
			\end{split}
			& \text{if } p=2n+1.\\
		\end{cases}
	\end{equation}
\end{thm}

\begin{proof}
	Substitute $t=-1$ in Theorem~\ref{JonesWp2}.
	Put $a_{n} = V_{W(2n,2)}(-1)$ and $b_{n} = V_{W(2n+1,2)}(-1)$ for $n=1,2,3,\ldots$.
	Then
	\begin{align*}
		& z = 2i,\\
		& a_1 = -2i, \quad a_2 = -12i, \quad b_1 = 5, \quad b_2 = 29,\\
		& a_{n} = a_{n-1} - 2i b_{n-1}, \quad b_{n} = b_{n-1} + 2i a_{n} \qquad (n\geq 3).
	\end{align*}
	This implies
	\begin{align*}
		\frac{a_{n+1} - a_{n}}{-2i} = \frac{a_{n} - a_{n-1}}{-2i} + 2i a_{n}
		& \quad \Rightarrow \quad a_{n+1} = 6a_{n} - a_{n-1} \quad (n\geq 2),\\
		\frac{b_{n} - b_{n-1}}{2i} = \frac{b_{n-1} - b_{n-2}}{2i} - 2i b_{n-1}
		& \quad \Rightarrow \quad b_{n}  = 6b_{n-1} - b_{n-2} \quad (n\geq 3).
	\end{align*}
	The characteristic equation of both the linear recurrence relations is $x^2-6x+1=0$.
	Its roots are $3+2\sqrt{2}$ and $3-2\sqrt{2}$.
	By solving these recurrence relations for the given initial conditions, we obtain 
	\begin{align*}
		a_n & = -\frac{i}{2\sqrt{2}} \left[ \left(3 + 2\sqrt{2}\right)^n - \left(3 - 2\sqrt{2}\right)^n \right]
		& (n = 1,2,3,\ldots),\\
		b_n & = \frac{1}{4} \left[ \left(2+\sqrt{2}\right) \left(3 + 2\sqrt{2}\right)^n +
		\left(2-\sqrt{2}\right) \left(3 - 2\sqrt{2}\right)^n \right]
		& (n = 1,2,3,\ldots).
	\end{align*}
	Since $\det (W(2n,2)) = \vert a_n \vert$ and $\det (W(2n+1,2)) = \vert b_n \vert$, we obtain~\eqref{eq:detWp2}.
\end{proof}

\section*{Appendix}
Using the formulae given by \eqref{eq:detWp2}, \eqref{eq:detW3n}, \eqref{eq:1eJonesWp2}, \eqref{eq:2eJonesWp2}, and \eqref{eq:eJonesW3n}, we provide a list containing some weaving knots with their respective knot determinants and values of Jones polynomial at $t=w=e^{i\pi/3}$ in Table~\ref{tab1}.
Few of these values can be verified from online databases~(see Livingston and Moore~\cite[KnotInfo-LinkInfo]{LivMoo}).
In Table~\ref{tab2}, we provide another list containing some weaving knots with their respective Jones polynomials calculated using Theorem~\ref{JonesWp2}.

\begin{table}[h!]
	\centering
	\caption{The determinant, $\det(K)$, and the value of Jones polynomial at $t=e^{i\pi/3}$, $V_K(w)$, of the knot $K=W(p,q)$ for certain values of $p$ and $q$.}\label{tab1}
	\begin{subtable}{0.45\textwidth}
		\centering
		\begin{tabular}{lrr}
			\toprule
			$K$ & $\det(K)$ & $V_K(w)$\\
			\midrule
			$W(2,2) = 2_1^2$  & $2$      & $-i$\\
			$W(3,2) = 4_1$    & $5$      & $-1$\\
			$W(4,2) = 6_3^2$  & $12$     & $\sqrt{3}$\\
			$W(5,2) = 8_{12}$ & $29$     & $-1$\\
			$W(6,2)$          & $70$     & $i$\\
			$W(7,2)$          & $169$    & $1$\\
			$W(8,2)$          & $408$    & $-\sqrt{3}$\\
			$W(9,2)$          & $985$    & $1$\\
			$W(10,2)$         & $2378$   & $-i$\\
			$W(11,2)$         & $5741$   & $-1$\\
			$W(12,2)$         & $13860$  & $\sqrt{3}$\\
			$W(13,2)$         & $33461$  & $-1$\\
			$W(14,2)$         & $80782$  & $i$\\
			$W(15,2)$         & $195025$ & $1$\\	
			\bottomrule
		\end{tabular}
	\end{subtable}
	\,
	\begin{subtable}{0.475\textwidth}
		\centering
		\begin{tabular}{lrr}
			\toprule
			$K$ & $\det(K)$ & $V_K(w)$\\
			\midrule
			$W(3,2) = 4_1$      & $5$       & $-1$\\
			$W(3,3) = 6_2^3$    & $16$      & $1$\\
			$W(3,4) = 8_{18}$   & $45$      & $3$\\
			$W(3,5) = 10_{123}$ & $121$     & $1$\\
			$W(3,6)$            & $320$     & $-1$\\
			$W(3,7)$            & $841$     & $1$\\
			$W(3,8)$            & $2205$    & $3$\\
			$W(3,9)$            & $5776$    & $1$\\
			$W(3,10)$           & $15125$   & $-1$\\
			$W(3,11)$           & $39601$   & $1$\\
			$W(3,12)$           & $103680$  & $3$\\
			$W(3,13)$           & $271441$  & $1$\\
			$W(3,14)$           & $710645$  & $-1$\\
			$W(3,15)$           & $1860496$ & $1$\\	
			\bottomrule
		\end{tabular}
	\end{subtable}
\end{table}

\renewcommand{\arraystretch}{1.3}
\begin{longtable}{lp{9cm}}
	\caption{The Jones polynomial of the knot $K=W(p,2)$ for certain values of $p$.}\label{tab2}\\
	\toprule
	$K$ & $V_K(t)$\\
	\midrule
	$W(2,2) = 2_1^2$	& $-t^{\frac{1}{2}} - t^{\frac{5}{2}}$\\
	$W(3,2) = 4_1$		& $t^{-2} - t^{-1} + 1 - t + t^2$\\
	$W(4,2) = 6_3^2$	& $-t^{-\frac{3}{2}} + 2t^{-\frac{1}{2}} - 2t^{\frac{1}{2}} + 2t^{\frac{3}{2}}
						- 3t^{\frac{5}{2}} + t^{\frac{7}{2}} - t^{\frac{9}{2}}$\\
	$W(5,2) = 8_{12}$	& $t^{-4} - 2t^{-3} + 4t^{-2} - 5t^{-1} + 5 - 5t + 4t^2 - 2t^3 + t^4$\\
	$W(6,2)$  			& $-t^{-\frac{7}{2}} + 3t^{-\frac{5}{2}} - 6t^{-\frac{3}{2}} +
							9t^{-\frac{1}{2}} - 11t^{\frac{1}{2}} + 12t^{\frac{3}{2}} -
							11t^{\frac{5}{2}} + 8t^{\frac{7}{2}} - 6t^{\frac{9}{2}} + 2t^{\frac{11}{2}} - t^{\frac{13}{2}}$\\
	$W(7,2)$  			& $t^{-6} - 3t^{-5} + 8t^{-4} - 14t^{-3} + 20t^{-2} - 25t^{-1}+	
						27 - 25t + 20t^2 - 14t^3 + 8t^4 - 3t^5 + t^6$\\
	$W(8,2)$			& $-t^{-\frac{11}{2}} + 4t^{-\frac{9}{2}} - 11t^{-\frac{7}{2}} + 
						22t^{-\frac{5}{2}} - 35t^{-\frac{3}{2}} + 48t^{-\frac{1}{2}} - 58t^{\frac{1}{2}} + 61t^{\frac{3}{2}} - 56t^{\frac{5}{2}} + 46t^{\frac{7}{2}} - 33t^{\frac{9}{2}} + 19t^{\frac{11}{2}} - 10t^{\frac{13}{2}} + 3t^{\frac{15}{2}} - t^{\frac{17}{2}}$\\
	$W(9,2)$			& $t^{-8} - 4t^{-7} + 13t^{-6} - 29t^{-5} + 53t^{-4} - 82t^{-3} +
						110t^{-2} - 131t^{-1} + 139 - 131t + 110t^2 - 82t^3 + 53t^4 - 29t^5 + 13t^6 - 4t^7 + t^8$\\
	\bottomrule
\end{longtable}

\bibliographystyle{plain}
\bibliography{refs}

\end{document}